\documentclass[11pt]{article}

\usepackage[a4paper,margin=1in]{geometry}
\usepackage{amsmath,amssymb,amsthm,mathtools}
\usepackage{microtype}
\usepackage{hyperref}
\usepackage{xcolor}
\usepackage{enumitem}
\usepackage{verbatim}
\usepackage{authblk}

\newtheorem{theorem}{Theorem}
\newtheorem{remark}[theorem]{Remark}
\newtheorem{lemma}[theorem]{Lemma}

\hypersetup{
  colorlinks=true,
  linkcolor=blue!50!black,
  citecolor=blue!50!black,
  urlcolor=blue!50!black
}

\DeclareMathOperator{\sgn}{sgn}
\newcommand{\Maj}{\mathrm{Maj}}

\title{\bf  Counterexample to\\
majority optimality in NICD with erasures}

\author[1]{Paata Ivanisvili}
\author[1]{Xinyuan Xie}

\affil[1]{Department of Mathematics, University of California, Irvine, CA 92697, USA\\\texttt{\{pivanisv,xinyuax7\}@uci.edu}}

\date{}

\begin{document}
\maketitle

\begin{abstract}
We asked GPT-5 Pro to look for counterexamples among a public list of open problems (the Simons ``Real Analysis in Computer Science'' collection). After several numerical experiments, it suggested a counterexample for the \emph{Non-Interactive Correlation Distillation (NICD) with erasures} question: namely, a Boolean function on \(5\) bits that achieves a strictly larger value of \(\mathbb{E}|f(z)|\) than the \(5\)-bit majority function when the erasure parameter is \(p=0.40\).
In this very short note we record the finding, state the problem precisely, give the explicit function, and verify the computation step by step \emph{by hand} so that it can be checked without a computer. In addition, we show that for each fixed odd \(n\) the majority is optimal (among unbiased Boolean functions) in a neighborhood of \(p=0\).
We view this as a ``little spark'' of an AI contribution in Theoretical Computer Science: while modern Large Language Models (LLMs) often assist with literature and numerics, here a concrete finite counterexample emerged.
\end{abstract}

\section*{A short preface}
Finding a successful instance in which an AI suggests a viable counterexample to a decade-plus-old open question appears, to our knowledge, to be the first publicly reported case. In most problems we tested, the models contributed little beyond literature pointers and small-scale numerics. Here, by contrast, the outcome is a clear win; had someone shown us this instance beforehand, we would have been impressed. The result also echoes the now-classical counterexamples \cite{Jain2017,Biswas23} related to ``Majority is Least Stable'' for linear threshold functions (LTFs): even if an AI merely applied a known LTF counterexample pattern in a new context and verified it, that would still merit credit.

\section{Problem statement}
We work on the hypercube \(\{-1,1\}^n\) equipped with the uniform probability measure. A Boolean function is \(f:\{-1,1\}^n\to\{-1,1\}\). We say \(f\) is \emph{unbiased} if \(\mathbb{E}[f]=0\). A sufficient condition is \emph{oddness}: \(f(-x)=-f(x)\).

\paragraph{Multilinear extension.}
Every \(f\) has a unique multilinear polynomial representation
\[
f(z)=\sum_{S\subseteq[n]}\widehat f(S)\,\prod_{i\in S} z_i, \qquad z\in[-1,1]^n,
\]
where \(\widehat f(S)=\mathbb{E}[f(x)\prod_{i\in S}x_i]\), and \(x \sim \mathrm{unif}(\{-1,1\}^{n})\).

\paragraph{Erasure model.}
Fix \(p\in(0,1)\). Let \(z\in\{-1,0,1\}^n\) have independent coordinates with
\[
\Pr[z_i=1]=\Pr[z_i=-1]=\tfrac{p}{2}, \qquad \Pr[z_i=0]=1-p.
\]
This is the standard \emph{binary erasure} NICD model (cf.\ Yang~\cite{Yang2007}).\footnote{Yang introduced and analyzed NICD for various noise models, including the erasure model; the maximization question we study below is how it is recorded in later Boolean-analysis sources and in the Simons open-problems list.}

\paragraph{Objective.}
For unbiased Boolean \(f\), define
\[
\Phi_p(f)\;:=\;\mathbb{E}_{z}\big[\,|f(z)|\,\big],
\]
where \(f(z)\) is the multilinear extension evaluated at the erasure sample \(z\).

\paragraph{Open question (as recorded in the Simons list).}
For \(p<\tfrac12\), is \(\Phi_p(f)\) maximized (over unbiased \(f\)) by a majority function? (When \(p\ge \tfrac12\), dictators are optimal; see \cite{ODWright2012,ODonnellOP}.) We refer to the Simons compilation \cite{SimonsOpen} for a concise statement.

\section{A finite counterexample at \texorpdfstring{$p=0.40$}{p=0.40}}
Set \(n=5\) and \(p=\frac{2}{5}=0.40\).
Consider the linear threshold function
\[
f(x_1,\ldots,x_5)\;=\;\sgn\!\big(x_1-3x_2+x_3-x_4+3x_5\big)
\]
(with the convention \(\sgn(t)=1\) for \(t>0\), \(-1\) for \(t<0\); note \(t\neq 0\) in our case).
Then \(f(-x)=-f(x)\), hence \(f\) is unbiased.

\begin{theorem}
For \(p=0.40\) one has
\begin{align*}
&\Phi_{0.40}(f)=\frac{2689}{6250}=0.43024,\quad
\Phi_{0.40}(\Maj_5)=\frac{5363}{12500}=0.42904,
\end{align*}
hence \(\Phi_{0.40}(f)> \Phi_{0.40}(\Maj_5)\). In particular, majority is not optimal at \(p=0.4\) already for \(n=5\).
\end{theorem}
\begin{proof}
See Appendix~A.
\end{proof}

\section{A related (solved) conjecture: Majority is least stable}
The ``Majority‑is‑Least‑Stable'' conjecture states: for all linear threshold functions $f:\{-1,1\}^n\to\{-1,1\}$, all odd $n$, and all $p \in [0,1]$, the noise stability
\[
\mathrm{Stab}_{p}[f]\ :=\ \sum_{S \subseteq [n]}p^{|S|}\,\widehat{f}(S)^2
\]
is minimized by the majority function $\Maj_n$.
Equivalently, with $z$ drawn from the erasure model at rate $p$, one has $\mathrm{Stab}_{p}(f)=\mathbb{E}\big[f(z)^2\big]$, so the conjecture asserts that the majority functions minimize the \(L_2\)-norm of $f(z)$ among LTFs. By contrast, the NICD‑with‑erasures question asks whether majority \emph{maximizes} the \(L_1\)-norm $\Phi_p(f)=\mathbb{E}|f(z)|$ among unbiased $f$. Although \(\mathbb{E}|f(z)| \le \sqrt{\mathbb{E}f(z)^2}\), to the best of our knowledge there is no implication in either direction between these two optimization problems.
After identifying the counterexample presented here, we noted that it is closely related in spirit to the counterexample in \cite{Jain2017}, e.g.\ to $\sgn(2x_1+2x_2+x_3+x_4+x_5)$; both have the ``two heavy + three light'' LTF structure, and up to permutation and flips of variables, they are the same function.

\begin{remark}
Comparison between \(L_1\) norms is often more delicate than comparison between \(L_2\) norms due to the lack of orthogonality. This is reflected in Appendix~A: the manual verification for \(L_1\) occupies more casework than the corresponding \(L_2\) verification in \cite{Jain2017}. This may help explain why a counterexample familiar from the ``Majority‑is‑Least‑Stable'' setting had not, to our knowledge, been tested against the NICD‑with‑erasures question until now.
\end{remark}

\begin{remark}
In the Majority‑is‑Least‑Stable question (\(L_2\) norm), the counterexample works in a neighborhood of $p=0$. On the other hand, in the NICD‑with‑erasures question there is no counterexample in a neighborhood of $p=0$ for each fixed $n$. Indeed, in this case we show that majority is optimal (see Section~\ref{majopt1}).
\end{remark}

\begin{remark}
A brute-force search is feasible for this finite question; indeed, GPT-5 Pro noted in the chat history that it used such a method. Still, being able to program against an open-problems list and to uncover a concrete counterexample seems noteworthy to us.
\end{remark}

\section{For each fixed $n$ majority is optimal in a neighborhood of $p=0$.}\label{majopt1}

\begin{lemma}\label{lemma1}
For unbiased $f:\{\pm1\}^n\to\{\pm1\}$ we have
\[
\Bigl|\Phi_p(f)-p\sum_{i=1}^n|\hat f(i)|\Bigr|
\;\le\; \Bigl((n-1)\sqrt n+\tbinom{n}{2}\Bigr)p^2.
\]
\end{lemma}

\begin{proof}
Let $K:=|\{i:Z_i\neq 0\}|$, so $K\sim\mathrm{Bin}(n,p)$. By conditioning on $K$,
\[
\Phi_p(f)=\sum_{k=0}^n \Pr[K=k]\;\mathbb E\!\left[\,|f(Z)|\;\middle|\;K=k\right].
\]
We treat three cases.

\medskip
\noindent\textbf{Case \(K=0\).} Then $Z=0$ and $|f(Z)|=|\hat f(\emptyset)|=0$ since $f$ is unbiased.

\medskip
\noindent\textbf{Case \(K=1\).} We have
\[
\mathbb E\!\left[\,|f(Z)|\;\middle|\;K=1\right]
=\frac{1}{n}\sum_{i=1}^n |\hat f(i)|,
\]
hence the unconditional contribution from $K=1$ equals
\[
\Pr[K=1]\cdot \mathbb E\!\left[\,|f(Z)|\;\middle|\;K=1\right]
= p(1-p)^{n-1}\sum_{i=1}^n |\hat f(i)|.
\]
Therefore
\[
\bigl|\;p(1-p)^{n-1}-p\;\bigr|\sum_{i=1}^n|\hat f(i)|
\le p\bigl(1-(1-p)^{n-1}\bigr)\sum_{i=1}^n|\hat f(i)|
\le (n-1)\sqrt{n}\,p^2,
\]
using $1-(1-p)^{n-1}\le (n-1)p$ and $\sum_i|\hat f(i)|\le \sqrt n$ (by Cauchy--Schwarz).

\medskip
\noindent\textbf{Case \(K\ge2\).} Since $|f(Z)|\le 1$,
\[
\sum_{k\ge2}\Pr[K=k]\;\mathbb E\!\left[\,|f(Z)|\;\middle|\;K=k\right]
\;\le\;\Pr[K\ge2]
\;\le\; \sum_{1\le i<j\le n}\Pr[Z_i\neq 0,\,Z_j\neq 0]
=\binom{n}{2}p^2.
\]

\medskip
Combining the three cases gives
\[
\left|\Phi_p(f)-p\sum_{i=1}^n|\hat f(i)|\right|
\le (n-1)\sqrt n\,p^2+\binom{n}{2}p^2,
\]
as claimed.
\end{proof}

\begin{lemma}[\cite{ODbook21}, Thm.~2.33]\label{lemma2}
Let $n$ be odd. Among all Boolean functions $f:\{-1,1\}^n\to\{-1,1\}$, the quantity
\[
\sum_{i=1}^n \widehat{f}(i)
\]
is uniquely maximized by the majority function
\[
\Maj_n(x)\;=\;\sgn(x_1+\cdots+x_n).
\]
\end{lemma}

\begin{theorem}
Fix odd $n$. There exists $p_0=p_0(n)>0$ such that for all $p\in(0,p_0)$, the function $\Maj_n$ is the unique maximizer of $\Phi_p$ over unbiased Boolean $f$; i.e.,
\[
\Phi_p(\Maj_n) > \Phi_p(f)\qquad\text{for all unbiased }f\neq \Maj_n.
\]
\end{theorem}

\begin{proof}
By Lemma~\ref{lemma1},
\[
\Phi_p(g)=p\sum_{i=1}^n|\widehat g(i)|+E_g(p),\qquad |E_g(p)|\le C_n\,p^2,
\]
for every unbiased $g$, where $C_n:=(n-1)\sqrt n+\binom{n}{2}$.  
Since $\Phi_p$ and $\sum_i|\hat g(i)|$ are invariant under flipping individual coordinates, we may assume w.l.o.g.\ that all level‑1 coefficients are nonnegative for the competitor $f$, so $\sum_i|\widehat f(i)|=\sum_i\widehat f(i)$ and likewise for $\Maj_n$.

Let
\[
\delta_n\ :=\ \min_{f\ne \Maj_n}\ \Big(\sum_{i=1}^n \widehat{\Maj_n}(i)-\sum_{i=1}^n \widehat f(i)\Big)\ >\ 0,
\]
which is strictly positive by Lemma~\ref{lemma2} and finiteness of the domain \(\{-1,1\}^n\) for fixed \(n\) (uniqueness of the maximizer implies a positive gap). Then for any unbiased \(f\ne \Maj_n\),
\[
\Phi_p(\Maj_n)-\Phi_p(f)
= p\!\left(\sum_i\widehat{\Maj_n}(i)-\sum_i\widehat f(i)\right) + \bigl(E_{\Maj_n}(p)-E_f(p)\bigr)
\;\ge\; p\,\delta_n - 2C_n\,p^2.
\]
Choose \(p_0:=\min\{\delta_n/(4C_n),\,1\}\). Then for all \(p\in(0,p_0)\), we have \(p\,\delta_n-2C_n p^2\ge \tfrac12 p\,\delta_n>0\), proving the claim and the uniqueness.
\end{proof}

\section*{Acknowledgments}
P.I.\ thanks EpochAI for encouraging broad testing of models, and University of W\"urzburg for its warm hospitality. The P.I.\ was supported in part by NSF CAREER grant DMS-2152401 and a Simons Fellowship. This work benefited from interactions with ChatGPT (GPT-5 Pro, OpenAI), which was used to find the counterexample. Following arXiv policy, the model is acknowledged here but not listed as an author.

\appendix
\section*{Appendix A}
\small

\begin{verbatim}
The full chat conversation between P.I. and GPT-5 Pro (ChatGPT) is available here

        https://chatgpt.com/share/68e80fbd-1f94-8001-8eb3-32be3ccd1bf2
\end{verbatim}

\paragraph{Manual verification.}

\paragraph{Multilinear expansions of majority functions.}
The majority function $\Maj_n: \{-1,1\}^n \to \{-1,1\}$ is defined for all odd $n$ as
\[
\Maj_n(x_1,\dots,x_n)=\sgn(x_1+\cdots+x_n),
\]
i.e., $\Maj_n(x) = 1$ if a majority of the coordinates are $1$, and $-1$ otherwise. Its unique multilinear expansions for $n=3,5$ are
\begin{align*}
\Maj_3(x_1,x_2,x_3)&= \tfrac{1}{2}(x_1+x_2+x_3) - \tfrac{1}{2}x_1x_2x_3,\\
\Maj_5(x)
&=\tfrac{3}{8}\sum_{i=1}^5 x_i
-\tfrac{1}{8}\!\!\sum_{1\le i<j<k\le 5}\!\! x_i x_j x_k
+\tfrac{3}{8}\, x_1 x_2 x_3 x_4 x_5.
\end{align*}
See Theorem~5.19 in \cite{ODbook21} for the general formula for the Fourier coefficients of majority when $n$ is odd.

\paragraph{Multilinear expansion of the counterexample $f$.}
The counterexample found by GPT‑5 Pro,
\[
f(x_1,\ldots,x_5)\;=\;\sgn\!\big(x_1 - 3x_2+x_3 -x_4+3x_5\big),
\]
admits the following factorized multilinear form:
\[
f(x)=\tfrac14\Big(2(-x_2+x_5)+(1+x_2x_5)\,(x_1+x_3-x_4+x_1x_3x_4)\Big).
\]
Indeed, $f$ acts as follows: if $-x_2+x_5=0$, the output is the majority of $(x_1,x_3,-x_4)$; otherwise it outputs the majority of $(-x_2,x_5)$. Thus
\[
f(x)=\Maj_3\!\big(-x_2,\,x_5,\,\Maj_3(x_1,x_3,-x_4)\big).
\]
Plugging in $\Maj_3(a,b,c)=\tfrac12(a+b+c-abc)$ gives the displayed formula.

\paragraph{Expected value of $|\Maj_5(z)|$.}
Extend functions on $\{-1,1\}^5$ to $\{-1,0,1\}^5$ by evaluating their multilinear expansions at $z\in\{-1,0,1\}^5$. Let $z$ be drawn from the erasure model and set
\[
K=\#\{i:z_i\neq 0\}\sim\mathrm{Bin}(5,p),\qquad
\Pr(K=k)=\binom{5}{k}p^{k}(1-p)^{5-k}.
\]
Then
\begin{align*}
\mathbb{E}\,|\Maj_5(z)|
&=\sum_{k=0}^5 \mathbb{E}\!\left[\,|\Maj_5(z)|\mid K=k\right]\Pr(K=k)\\
&=0\cdot(1-p)^5+\tfrac{3}{8}\cdot 5p(1-p)^4+\tfrac{3}{8}\cdot 10p^2(1-p)^3\\
&\quad+\tfrac{5}{8}\cdot 10p^3(1-p)^2+\tfrac{5}{8}\cdot 5p^4(1-p)+1\cdot p^5\\
&=\tfrac{15}{8}p-\tfrac{15}{4}p^2+\tfrac{25}{4}p^3-\tfrac{45}{8}p^4+\tfrac{9}{4}p^5.
\end{align*}
Here are the conditional expectations, by symmetry (we may fix which coordinates are nonzero and average over their signs). \textbf{Case \(K=0\)}: \(z=(0,0,0,0,0)\), so \(|\Maj_5(z)|=0\).
\textbf{Case \(K=1\)}: without loss of generality (by symmetry) $z_1\neq0$, so \(|\Maj_5(z)|=\tfrac{3}{8}\).
\textbf{Case \(K=2\)}: assume $z_1,z_2\neq0$, then $\Maj_5(z)=\tfrac{3}{8}(z_1+z_2)$, hence the conditional expectation is $\tfrac{3}{8}$.
\textbf{Case \(K=3\)}: assume $z_1,z_2,z_3\neq0$, then $\Maj_5(z)=\tfrac{3}{8}(z_1+z_2+z_3)-\tfrac{1}{8}z_1z_2z_3$; a check over $\{\pm1\}^3$ gives $\tfrac{5}{8}$.
\textbf{Case \(K=4\)}: assume $z_1,\dots,z_4\neq0$, $z_5=0$. Writing $m$ for the number of $-1$'s among $(z_1,\dots,z_4)$, one gets $|\Maj_5(z)|=0$ iff $m=2$ (6 of 16 points) and $=1$ otherwise; hence $\tfrac{5}{8}$.
\textbf{Case \(K=5\)}: trivially the value is $1$.

\paragraph{Expected value of $|f(z)|$.}
Set
\[
M:=\tfrac12\!\big(z_1+z_3-z_4+z_1z_3z_4\big).
\]
Using the factorized form,
\[
f(z)=\tfrac12(-z_2+z_5)+\tfrac12(1+z_2z_5)M .
\]
Taking the conditional expectation over $(z_2,z_5)$ yields
\[
\mathbb{E}|f(z)|
=\frac{1-2p+2p^2}{2}\,\mathbb{E}|M|+p-\frac{p^2}{2},
\]
since the four cases are: both zero (prob.\ $(1-p)^2$), giving $\tfrac12|M|$; exactly one zero (prob.\ $2p(1-p)$), giving $\tfrac14(|M+1|+|M-1|)=\tfrac12$ because $M\in[-1,1]$; both nonzero and equal (prob.\ $p^2/2$), giving $|M|$; and both nonzero and opposite (prob.\ $p^2/2$), giving $1$.

It remains to compute $\mathbb{E}|M|$ by the number of nonzeros among $(z_1,z_3,z_4)$: with probability $(1-p)^3$ there are none and $|M|=0$; with probability $3p(1-p)^2$ there is exactly one, yielding $|M|=\tfrac12$; with probability $3p^2(1-p)$ there are exactly two, again averaging to $\tfrac12$; and with probability $p^3$ there are three, in which case $M$ is the majority of three bits and $|M|=1$. Therefore,
\[
\mathbb{E}|M|=\tfrac32p-\tfrac32p^2+p^3,
\]
and hence
\[
\mathbb{E}|f(z)|
=\frac{7}{4}p-\frac{11}{4}p^2+\frac{7}{2}p^3-\frac{5}{2}p^4+p^5.
\]
At $p=0.4$ this gives $\mathbb{E}|f(z)|=\frac{2689}{6250}=0.43024$, whereas the earlier computation yields $\mathbb{E}\,|\Maj_5(z)|=\frac{5363}{12500}=0.42904$.

\normalsize

\bibliographystyle{abbrv}

\end{document}